\documentclass{amsart}
\usepackage{latexsym}
\usepackage{amsmath}
\usepackage{amssymb}
\usepackage[all]{xy}
\usepackage{autobreak}
\usepackage{enumitem}
\usepackage[hypertexnames=false]{hyperref} 
\usepackage[capitalize]{cleveref}
\crefname{equation}{}{}
\crefname{enumi}{}{}

\usepackage{tikz}
\usetikzlibrary{positioning}
\usetikzlibrary{arrows.meta}

\numberwithin{equation}{section}
\newtheorem{thm}{Theorem}[section]
\newtheorem{prop}[thm]{Proposition}
\newtheorem{cor}[thm]{Corollary}
\newtheorem{lem}[thm]{Lemma}

\theoremstyle{definition}
\newtheorem{defn}[thm]{Definition}

\theoremstyle{remark}
\newtheorem{rem}[thm]{Remark}
\newtheorem{ex}[thm]{Example}

\newcommand{\K}{{\mathbb K}}
\newcommand{\Q}{{\mathbb Q}}
\newcommand{\R}{{\mathbb R}}
\newcommand{\mapright}[1]{%
 \smash{\mathop{%
  \hbox to 1cm{\rightarrowfill}}\limits_{#1}}}
\newcommand{\maprightd}[2]{%
 \smash{\mathop{%
  \hbox to 0.5cm{\rightarrowfill}}\limits^{#1}\limits_{#2}}}
\newcommand{\mapleft}[1]{%
 \smash{\mathop{%
  \hbox to 1cm{\leftarrowfill}}\limits_{#1}}}
\newcommand{\mapleftu}[1]{%
 \smash{\mathop{%
  \hbox to 0.8cm{\leftarrowfill}}\limits^{#1}}}
\newcommand{\maprightu}[1]{%
 \smash{\mathop{%
  \hbox to 1cm{\rightarrowfill}}\limits^{#1}}}
\newcommand{\maprightud}[2]{%
 \smash{\mathop{%
  \hbox to 1cm{\rightarrowfill}}\limits^{#1}_{#2}}}
\newcommand{\mapleftud}[2]{%
 \smash{\mathop{%
  \hbox to 1cm{\leftarrowfill}}\limits^{#1}_{#2}}}

\begin{document}
\title[On F\'elix--Tanr\'e rational models for polyhedral products]{On F\'elix--Tanr\'e rational models for polyhedral products 
}

\footnote[0]{ 



Department of Mathematical Sciences, 
Faculty of Science,  
Shinshu University,   
Matsumoto, Nagano 390-8621, Japan   
e-mail:{\tt kuri@math.shinshu-u.ac.jp}
}

\author{Katsuhiko KURIBAYASHI}
\date{}
   
\maketitle

\begin{abstract} 
The F\'elix--Tanr\'e rational model for the polyhedral product of a fibre inclusion is considered. 
In particular, we investigate the rational model for the polyhedral product of a pair of Lie groups corresponding to arbitrary simplicial complex 
and the rational homotopy group of the polyhedral product.  
Furthermore, it is proved that for a partial quotient $N$ associated with a toric manifold $M$, the following conditions are equivalent: (i) $N=M$. (ii) The odd-degree rational cohomology of $N$ is trivial. (iii) The torus bundle map from $N$ to the Davis--Januszkiewicz space is formalizable. 
\end{abstract}

\section{introduction}\label{sect:Introduction}
Toric varieties are fascinating objects in the study of algebraic geometry, combinatorics, symplectic geometry and topology.  For a nonsingular toric variety, so-called a toric manifold, is given by the quotient of a moment-angle manifold by a torus action 
with Cox's construction.  By generalizing the construction of moment-angle manifolds, we obtain 
a {\it moment-angle complex} and more general {\it polyhedral products} \cite{BBCG, G-T, K-L}, which are defined by the colimit of spaces with gluing data obtained from a simplicial complex.  
Thus we are also interested in the generalized ones. 

In \cite{F-T}, F\'elix and Tanr\'e have given a rational model for a polyhedral product of a tuple of spaces corresponding to an arbitrary simplicial complex. One of the aims of this manuscript is to construct a tractable rational model for a polyhedral product by refining the model due to F\'elix and Tanr\'e. By applying the construction to a polyhedral product for a pair of Lie groups, we have a result on the rational homotopy groups of the polyhedral product; see Theorem \ref{thm:Rational_homotopy} and Proposition \ref{prop:PolyProd-model}. 

Moreover, the formality of a toric manifold and the non-formalizability for a partial quotient, which is not a toric manifold, are discussed with their models induced by the F\'elix and Tanr\'e rational models for polyhedral products; see Theorems \ref{thm:partial_quotients}, \ref{thm:formality} and \ref{thm:BorelCoh} for more details.

Throughout this article, each space $X$ is assumed to be connected and ($\Q$-)locally finite; that is, the rational cohomology group $H^i(X; \Q)$ is of finite dimension 
for $i\geq 0$.  
In the rest of this section, we describe our main results more precisely.  
Following Kishimoto and Levi \cite{K-L}, we define a polyhedral product with the homotopy colimit instead of the colimit; see also \cite{N-R} for the study of the Davis-Januszkiewicz space with the homotopy colimit functor.

\begin{defn}\label{defn:PolyProd}\text{(\cite[Definition 1.2]{K-L})}
Let 
$(\underline{X}, \underline{A}):= ((X_1, A_1), ..., (X_m, A_m))$ be a tuple of spaces with $A_i \subset X_i$ for each $i$ and 
$K$ a simplicial complex with $m$.  
The {\it polyhedral product} $(\underline{X}, \underline{A})^K$ of the tuple $(\underline{X}, \underline{A})$ corresponding to $K$ is defined by 
$$(\underline{X}, \underline{A})^K:= \displaystyle{\text{hocolim}_{\sigma \in K}(\underline{X}, \underline{A})^\sigma},$$ where 
$(\underline{X}, \underline{A})^\sigma = Y_1\times \cdots \times Y_m$ with 
\[
Y_i = \left\{
\begin{array}{ll}
A_i & i \notin \sigma \\
X_i & i \in \sigma .
\end{array}
\right.
\]
We write $(X, A)^K$ for $(\underline{X}, \underline{A})^K$ if there are a space $X$ and a subspace $A$ such that $X_i = X$ and $A_i =A$ for each $i$. 
\end{defn}

In what follows, we assume that a simplicial complex $K$ has no ghost vertices unless otherwise specified. 

Suppose that each $(X_i, A_i)$ is a pair of CW-complexes. Then, the natural map $\text{colim}_{\tau \in \partial(\sigma)}(\underline{X}, \underline{A})^\tau \to 
(\underline{X}, \underline{A})^\sigma$ is a cofibration. Thus, in view of \cite[\S2 and Proposition 4.8]{P-R} and also \cite[Proposition 8.1.1]{B-P}, we have a weak homotopy equivalence $$(\underline{X}, \underline{A})^K \stackrel{\simeq_w}{\longrightarrow} 
\text{colim}_{\sigma \in K}(\underline{X}, \underline{A})^\sigma = \bigcup_{\sigma \in K}(\underline{X}, \underline{A})^\sigma =: 
\mathcal{Z}_K(\underline{X}, \underline{A}).$$
In particular, by definition, the moment-angle complex $\mathcal{Z}_K(D^2, S^1)$ corresponding to a simplicial complex $K$ 
is the colimit $\bigcup_{\sigma \in K}(D^2, S^1)^\sigma$ and then it is weak homotopy equivalent to the polyhedral product 
$(D^2, S^1)^K$. 

Our first result is concerned with the rational homotopy groups of a polyhedral product of a pair of Lie groups. We denote by $\pi_*(X)_\Q$ the rational homotopy group $\pi_*(X)\otimes \Q$ for a pointed connected space $X$ whose fundamental group is abelian. 

\begin{thm}\label{thm:Rational_homotopy}
Let $G$ be a connected compact Lie group and $i : H \to G$ the inclusion of a maximal rank subgroup. Suppose that $G/H$ is simply connected and $(Bi)^*(x_k)$ is decomposable in $H^*(BH; \Q)$ for each generator $x_k$ of $H^*(BG; \Q)$. Then, one has a short exact sequence of rational homotopy groups 
\[
\xymatrix@C20pt@R12pt{
0 \ar[r] &  \pi_*((G, H)^K)_\Q \ar[r]^{q_*} & \pi_*((G/H, \ast)^K)_\Q \ar[r]^-{\partial_*} & \pi_{*-1}(\Pi^m H)_\Q \ar[r] & 0
}
\]
for arbitrary simplicial complex $K$ with $m$ verticies, 
where $\partial_*$ denotes the connecting homomorphism of the homotopy exact sequence of the middle vertical sequence in (\ref{eq:fibrations}) below.
\end{thm}

We stress that the exactness in the theorem above does not depend on any property of the given simplicial complex $K$. 

\begin{rem}
While we do not pursue topics on the cohomology $H^*((G, H)^K; \mathbb{K})$ with coefficients in arbitrary field $\mathbb{K}$, in order to compute the cohomology algebra, 
we may use a commutative diagram 
\begin{eqnarray}\label{eq:fibrations}
\xymatrix@C15pt@R15pt{
(H, H)^K \ar[d] &(H, H)^K  \ar@{=}[r] \ar@{=}[l] \ar[d] & \Pi^m  H  \ar[d] \\
(EG, H)^K \ar[d] & (G, H)^K \ar[r] \ar[l]  \ar[d]^-q&  \Pi^m G \ar[d] \\
(BH, \ast)^K &  (G/H, \ast)^K \ar[r] \ar[l]_j  &  \Pi^m G/H. 
}
\end{eqnarray}
in which vertical sequences are fibrations; see \cite[Lemma 2.3.1]{D-S}.  We can regard the lower squares as pullback diagrams. 
\end{rem}


Before describing our main result on a partial quotient, we recall some terminology in rational homotopy theory. 

A {\it commutative differential graded algebra} (henceforth, CDGA) $(A, d)$ consists of a non-negatively graded algebra $A$ and a differential $d$ on $A$ with degree $+1$. 
Let $A_{PL}(X)$ be the CDGA 
of polynomial differential forms on a space $X$; see \cite[10 (c)]{F-H-T}. 
It is worthwhile mentioning that there exists a morphism of cohain complexes from $A_{PL}(X)$ to the singular cochain algebra of $X$ with coefficients in the rational field $\Q$ which induces an isomorphism of algebras between cohomology algebras; see \cite[10(e) Remark]{F-H-T}. 

By definition, a {\it Sullivan algebra} $(A, d)$ is a CDGA whose underlying algebra $A$ is the free algebra $\wedge W$ generated by a graded vector space $W$ and for which the vector space $W$ admits a filtration 
$W_0 \subset W_1 \subset \cdots W_n \subset \cdots $ with $W= \bigcup_i W_i$, $d(W_0) = 0$ and $d : W_k \to \wedge W_{k-1}$ for $k\geq 1$. 
We say that a Sullivan algebra $(\wedge W, d)$ is {\it minimal} if $d(w)$ is decomposable for each $v \in W$. 

A morphism $\varphi : (A, d) \to (B, d')$ of CDGA's is a {\it quasi-isomorphism} if $\varphi$ induces an isomorphism on cohomology. 
A {\it rational model} $(A, d)$ for a space $X$ is a CDGA which is connected with $A_{PL}(X)$ by using quasi-isomorphisms. We call the rational model $(A, d)$ a {\it (minimal) Sullivan model} for $X$ if it is a (minimal) Sullivan algebra. Observe that each space has a unique minimal Sullivan model; see \cite[14(b) Corollary]{F-H-T}.
A space $X$ is {\it formal} if there exists a sequence of quasi-isomorphisms between a Sullivan model for $X$ and the cohomology $H^*(X; \Q)$ which is regarded as a CDGA with zero differential.   
We refer the reader to the books \cite{G-M}, \cite{F-H-T} and \cite{F-O-T} for rational homotopy theory and its applications to topology and geometry.

\begin{defn}\label{defn:formalizability}\text{(cf. \cite[V]{Thomas})} 
A map $p : E \to B$ is {\it formalizable} if  there exists a commutative diagram up to homotopy 
\[
\xymatrix@C25pt@R15pt{
A_{PL}(B) \ar[r]^{A_{PL}(p)} & A_{PL}(E) \\
(\wedge W, d) \ar[r]^l \ar[d]_{\simeq} \ar[u]^{\simeq}& (\wedge Z, d') \ar[d]^{\simeq}  \ar[u]_{\simeq} \\
H^*(B; \Q)  \ar[r]_{p^*} & H^*(E; \Q) 
}
\]
in which $(\wedge W, d)$ and $(\wedge Z, d')$ are minimal Sullivan algebras and vertical arrows are quasi-isomorphisms; 
see \cite[12 (b)]{F-H-T} and \cite[Chapter 5]{H} for the homotopy relation.  
\end{defn}

For a simplicial complex $K$, define $\mathcal{Z}_K({\mathbb C}, {\mathbb C}^*)$ by the colimit 
$\text{colim}_{\tau\in K}({\mathbb C}, {\mathbb C}^*)^\tau$. 
Then, we have weak equivalences 
$$
(D^2, S^1)^K \stackrel{\simeq_w}{\longrightarrow} 
\text{colim}_{\tau \in K}(D^2, S^2)^\tau \maprightd{i}{\simeq}  \mathcal{Z}_K({\mathbb C}, {\mathbb C}^*), 
$$
where $i$ is the inclusion; see 
\cite[Theorem 4.7.5]{B-P}.
Let $X_\Sigma$ be a compact toric manifold associated with a complete and smooth fan $\Sigma$; see \cite[\S 3.1]{CLS}. We then have a homeomorphism 
$X_\Sigma \cong \mathcal{Z}_K({\mathbb C}, {\mathbb C}^*)/H$ 
via Cox's construction of the manifold, where $K$ is the simplicial complex with $m$ vertices associated with the fan $\Sigma$ and $H$ is a subgroup of the torus $({\mathbb C}^*)^m$ which acts on  $\mathcal{Z}_K({\mathbb C}, {\mathbb C}^*)$ canonically and freely; see \cite[Theorem 5.1.11]{CLS} and \cite[Theorem 5.4.5, Proposition 5.4.6]{B-P}. 
Moreover, the quotient $\mathcal{Z}_K({\mathbb C}, {\mathbb C}^*)/H'$ by a subtorus $H' \subset H$ is called a {\it partial quotient}. 

We recall the pullback diagram in \cite[The proof of Proposition 3.2]{Franz}. Let $X_\Sigma$ be a toric manifold associated with a fan $\Sigma$ and 
$\mathcal{Z}_K({\mathbb C}, {\mathbb C}^*)/H$ Cox's construction of $X_\Sigma$ mentioned above. Then, we have a commutative diagram consisting of two pullbacks 
\begin{eqnarray}\label{eq:pullbacks}
\xymatrix@C20pt@R15pt{
EG \times_H  \mathcal{Z}_K({\mathbb C}, {\mathbb C}^*) \ar[r]^-p\ar[d]_{\pi_H}&  (EG)/H \ar[r]\ar[d]& EL\ar[d] \\
EG \times_{G}  \mathcal{Z}_K({\mathbb C}, {\mathbb C}^*) \ar[r]_-q&  BG \ar[r]_{B\rho}& BL, 
}
\end{eqnarray}
where $G = ({\mathbb C}^*)^m$ and $L =({\mathbb C}^*)^m/H$. We observe that right two vertical maps are principal $L$-bundles and that the maps 
$p$ and $q$ are fibrations associated with the universal $H$-bundle and the universal $G$-bundle, respectively. 
Since the group $H$ acts on $\mathcal{Z}_K({\mathbb C}, {\mathbb C}^*)$ freely, it follows that the Borel construction $EG \times_H  \mathcal{Z}_K({\mathbb C}, {\mathbb C}^*)$ is homotopy equivalent to the toric manifold  $X_\Sigma$. 

Let $H'$ be a subtorus of $H$. 
Then, we may replace $H$ and $L$ in the diagram (\ref{eq:pullbacks}) with $H'$ and $L':= ({\mathbb C}^*)^m/H$, respectively. 
With the replacement,  the upper left corner in the diagram is regarded as the partial quotient 
$\mathcal{Z}_K({\mathbb C}, {\mathbb C}^*)/{H'}\simeq EG \times_{H'}\mathcal{Z}_K({\mathbb C}, {\mathbb C}^*)$.  

\begin{rem} It follows from \cite[Theorems 4.3.2 and 4.7.5]{B-P} that the Borel construction $EG \times_{G}  \mathcal{Z}_K({\mathbb C}, {\mathbb C}^*)$ is homotopy equivalent to the Davis--Januszkiewicz space $DJ(K):= (BS^1, *)^K$. 
Since the fan that we consider is complete, it follows from the result \cite[Theorem 12.1.10]{CLS} that $X_\Sigma$ is simply connected. Then, we have an exact sequence 
$
\xymatrix@C15pt@R10pt{
0 \ar[r] &  \pi_*(X_\Sigma) \ar[r]^-{(\pi_H)_*} & \pi_*(DJ(K)) \ar[r]^-{\partial_*} & \pi_{*-1}(G/H) \ar[r] & 0. 
}
$
By considering the center vertical fibration mentioned in (\ref{eq:fibrations}), 
the exact sequence in Theorem \ref{thm:Rational_homotopy} is regarded as an analogy of the sequence above. 
\end{rem}

The following result characterizes a toric manifold among partial quotients associated with the manifold. 

\begin{thm}\label{thm:partial_quotients} Let $\mathcal{Z}_K({\mathbb C}, {\mathbb C}^*)/H$ be a toric manifold and $H'$ a subtorus of $H$. 
For the partial quotient $\mathcal{Z}_K({\mathbb C}, {\mathbb C}^*)/H'$, the following conditions are equivalent. 
\begin{itemize}
\item[\em (i)] $H = H'$. 
\item[\em (ii)] $H^{\text{\em odd}}(\mathcal{Z}_K({\mathbb C}, {\mathbb C}^*)/H' ; \Q) =0$.
\item[\em (iii)] The map $\pi_{H'} : \mathcal{Z}_K({\mathbb C}, {\mathbb C}^*)/H' \to DJ(K)$  in the diagram (\ref{eq:pullbacks}) is formalizable. 
\end{itemize}
\end{thm}

\begin{rem}
As seen in Theorem \ref{thm:formality}, a toric manifold is formal. However, we do not know whether a general partial quotient is formal.  
\end{rem}

An outline for the article is as follows. Section \ref{sect:F-T} recalls the construction of the F\'elix--Tanr\'e rational model for a polyhedral product and 
discusses the naturality of the models. In Section \ref{sect:TractableModels}, we give a tractable rational model for a polyhedral product and some examples for the model. 
Section \ref{sect:Liegroups} constructs a rational model for the polyhedral product $(G, H)^K$ of a pair of Lie group and closed subgroup corresponding to arbitrary simplicial complex $K$. With the model, we prove 
Theorem \ref{thm:Rational_homotopy}. In Section \ref{sect:formality}, we show that every compact toric manifold is formal. Section \ref{sect:formalizability} is devoted to proving Theorem \ref{thm:partial_quotients}.


\section{A recollection of the F\'elix--Tanr\'e rational models for polyhedral products}\label{sect:F-T}

 While the construction of a rational model in \cite{F-T} for a polyhedral product is defined by the colimit construction, 
 it is also applicable in constructing a rational model for $(\underline{X}, \underline{A})^K$ obtained 
 by the homotopy colimit as in Definition \ref{defn:PolyProd}. 
In this section, we summarize the result. 

Let $\iota_j : A_j \to X_j$ be the inclusion and $\varphi_j : \mathcal{M}_j \to \mathcal{M}_j'$ a surjective model\footnote{The existence of the model: We consider a Sullivan representative for $\iota_j$; see \cite[page 154]{F-H-T}.  The proof of \cite[Lemma 13.4]{F-H-T}  enables us to replace the homotopy commutative diagram of the representative with a strictly commutative diagram. By applying the surjective trick (\cite[\S 12 (b)]{F-H-T}), we have a surjective model for the inclusion.} for $\iota_j$, namely, an epimorphism of CDGA's which fits in  a commutative diagram
\begin{eqnarray}\label{eq:SurjectiveModels}
\xymatrix@C20pt@R15pt{
\wedge W_j  \ar[r]^-{\varphi_j} \ar[d]_{u}^\simeq & \wedge V_j  \ar[d]^v_\simeq\\
A_{PL}(X_j) \ar@{->>}[r]_{\iota_j^*} & A_{PL}(A_j)
}
\end{eqnarray}
with quasi-isomorphisms $u$ and $v$. We observe that $\iota_j^*$ is surjective; see \cite[Proposition 10.4 , Lemma 10.7]{F-H-T}. 
For each $\tau \notin K$, 
let $I_\tau$ denote the ideal of $\bigotimes_{i=1}^m\mathcal{M}_i$ defined by $I_\tau = E_1\otimes \cdots \otimes E_m$, where 
\[
E_i = \left\{
\begin{array}{ll}
\mathcal{M}_i & i \notin \tau \\
\text{Ker} \ \varphi_i & i \in \tau .
\end{array}
\right.
\]
 
\begin{thm}\text{\em (\cite[Theorem 1]{F-T})}\label{thm:F-T}
There is a sequence of quasi-isomorphisms connecting the CDGA $A_{PL}((\underline{X}, \underline{A})^K)$ and the quotient $(\bigotimes_{i=1}^m\mathcal{M}_i)/ J(K)$, where 
$J(K) := \sum_{\tau \notin K} I_\tau$; that is, the quotient is a rational model for $(\underline{X}, \underline{A})^K$. 
\end{thm}

In what follows, we may call the the quotient CDGA in Theorem  \ref{thm:F-T} the {\it F\'elix--Tanr\'e (rational) model} for the polyhedral product $(\underline{X}, \underline{A})^K$. 

\begin{rem}\label{rem:F-T}  
We observe that the polyhedral product $(\underline{X}, \underline{A})^K$ is defined by the homotopy colimit on the diagram associated to the simplicial complex $K$. 
While the nilpotency of each space in the pairs $(X_i, A_i)$ of CW complexes for $1 \leq i \leq m$ is assumed in \cite[Theorem 1]{F-T}, 
the conditions is not required in Theorem \ref{thm:F-T}. 
In fact, for each inclusion $\iota_j : A_j \to X_j$, we have a commutative diagram 
\begin{eqnarray}\label{eq:replacement}
\xymatrix@C20pt@R15pt{
A_j \ar[r]^-{\iota_j} & X_j \\
|S(A_j)| \ar[r]_{|S(\iota)|}\ar[u]^{\simeq}& |S(X_j)| \ar[u]_{\simeq}
}
\end{eqnarray}
with the singular simplex functor $S( \ )$ and the realization functor $| \ |$. Then, this enables us to obtain a sequence of weak homotopy equivalences 
\begin{eqnarray}\label{eq:homotopy_colimit}
\xymatrix@C20pt@R15pt{
 \hspace{1.0cm} (\underline{X}, \underline{A})^K & (\underline{X'}, \underline{A'})^K \ar[l]_-{\simeq_w} \ar[r]^-{\simeq_w} &
\text{colim}_{\sigma \in K}(\underline{X'}, \underline{A'})^\sigma = \bigcup_{\sigma \in K}(\underline{X'}, \underline{A'})^\sigma,
}
\end{eqnarray}
where each pair $(X_i', A_i')$ denotes the pair $( |S(X_j)|, |S(A_j)| )$; see the paragraph after Definition \ref{defn:PolyProd}. 
A surjective model for each inclusion $A_j \to X_j$ is regarded as that for the inclusion $A_j' \to X_j'$. Thus, with the models and by applying  \cite[Proposition 13.5]{F-H-T}  inductively as in the proof of \cite[Theorem 1]{F-T}, we can prove Theorem \ref{thm:F-T} without assuming that spaces 
$X_i$ and $A_i$ are CW-complexes and nilpotent. 
\end{rem}

In order to confirm the naturality of the model in Theorem \ref{thm:F-T} with respect to an inclusion of simplicial complexes and also given surjective models, 
the outline of the proof of Theorem \ref{thm:F-T} is recalled below.  

By using surjective models $\varphi_j : \mathcal{M}_j \to \mathcal{M}_j'$, for each $\sigma \in K$, 
we have a CDGA $\widetilde{D}^\sigma := \bigotimes_{i \in \sigma} \mathcal{M}_i \otimes \bigotimes_{i \notin \sigma} \mathcal{M}_i'$
and a map $\xi_\sigma : (\bigotimes_{i=1}^m\mathcal{M}_i)/ J(K) \to \widetilde{D}^\sigma$ of CDGA's defined by 
\[
\xi_\sigma(x_i ) = \left\{
\begin{array}{ll}
x_i & i \in \sigma \\
\varphi_i(x_i)   & i \notin \sigma .
\end{array}
\right.
\]
It is readily seen that $\xi_\sigma$ is well defined. 
The induction argument in the proof of \cite[Theorem 1]{F-T} yields that the maps $\xi_\sigma$ of CDGA's give rise to a quasi-isomorphism 
$$\alpha : (\bigotimes_{i=1}^m\mathcal{M}_i)/ J(K) \stackrel{\simeq}{\to} \text{lim}_{\sigma \in K}\widetilde{D}^\sigma.$$ 
We also observe that the fact is proved by using \cite[Lemma 13.3]{F-H-T} which gives a well-defined quasi-isomorphism between appropriate pullback diagrams in the category of CDGA's. 
Thus, with the same notation as in Remark \ref{rem:F-T}, we have a sequence of quasi-isomorphisms
\begin{eqnarray}\label{eq:The_sequence}
\xymatrix@C15pt@R15pt{ 
A_{PL}((\underline{X}, \underline{A})^K) \ar[r]^-{\simeq} & A_{PL}^*((\underline{X'}, \underline{A'})^K) & A_{PL}(\text{colim}_{\sigma \in K}(\underline{X'}, \underline{A'})^\sigma) \ar[lld]^-{\eta}_-{\simeq}  \ar[l]_-{\simeq}\\
\text{lim}_{\sigma \in K} A_{PL}((\underline{X'}, \underline{A'})^\sigma)&
\text{lim}_{\sigma \in K}\widetilde{D}^\sigma \ar[l]_-{\Phi}^-{\simeq} & (\bigotimes_{i=1}^m\mathcal{M}_i)/ J(K) \ar[l]_-{\alpha}^-{\simeq}
}
\end{eqnarray}
in which the first two quasi-isomorphisms are induced by the weak equivalences in (\ref{eq:homotopy_colimit}), $\Phi$ is defined by the surjective models 
$\varphi_i$ and $\eta$ is induced by natural maps $(\underline{X}, \underline{A})^\sigma \to (\underline{X}, \underline{A})^K$. 
It follows from \cite[Proposition 4.8]{P-R} and \cite[Proposition 8.1.4]{B-P} that $\Phi$ and $\eta$ are quasi-isomorphisms, respectively. 
This enables us to obtain the rational model for the polyhedral product $(\underline{X}, \underline{A})^K$ in Theorem \ref{thm:F-T}. Moreover, 
the construction above of the model yields the following proposition. 

\begin{prop}\label{prop:naturality}
The F\'elix--Tanr\'e rational models are natural with respect to surjective models which are used in constructing the models of polyhedral products and 
an inclusion of simplicial complexes.  
 \end{prop}

The rational cohomology of the moment-angle complex $\mathcal{Z}_K(D^2, S^1)$ is isomorphic to the torsion product 
$\text{Tor}_{\Q[t_1, ..., t_m]}(\Q[t_1, ..., t_m]/I(K), \Q)$, where $\deg t_i =2$ and $I(K)$ denotes  
 the ideal generated by monomials $t_{i_1}\cdots t_{i_s}$ for $\{i_1, ..., i_s\}\notin K$, 
 which is called the {\it Stanley--Reisner ideal} associated with $K$; see \cite[10.1]{Franz}. 
Thus, a CDGA of the form 
\begin{eqnarray}\label{eq:MAC}
(\wedge (x_i, ..., x_m)\otimes \Q[t_1, \cdots, t_m]/I(K),  d(x_i) = t_i)
\end{eqnarray}
computes the cohomology algebra $H^*(\mathcal{Z}_K(D^2, S^1); \Q)$. Here  $SR(K)$ denotes the {\it Stanley--Reisner algebra} $\Q[t_1, \cdots, t_m]/I(K)$. 

\begin{rem}\label{rem:RS} 
The inclusion $i : S^1 \to D^2$ admits a surjective model of the form $\pi : (\wedge (x, t), d) \to (\wedge (x), 0)$, where $\pi$ is the projection, $d(x) = t$ and $\deg x = 1$. By virtue of Theorem \ref{thm:F-T}, we see that the CDGA (\ref{eq:MAC}) above 
is a rational model for  $\mathcal{Z}_K(D^2, S^1)$. 
\end{rem}

\begin{ex}\label{ex:DJ}
Let $K$ be a simplicial complex with $m$ vertices and $j : K \to 2^{[m]}$ the inclusion. The map $j$ induces the inclusion 
$\widetilde{j} : (BS^1, *)^K \to \Pi^m(BS^1)$. We choose the projection $(\wedge (t), 0) \to \Q$ as a surjective model for the inclusion $* \to BS^1$, 
where $\deg t_i = 2$ for $i = 1, ..., m$. By Theorem \ref{thm:F-T} and 
Proposition \ref{prop:naturality}, we have a model $(\wedge(t_1, ..., t_m), 0) \to (\wedge(t_1, ..., t_m)/I(K), 0)= (SR(K), 0)$  
for $\widetilde{j}$ which is the natural projection. As a consequence, we see that the inclusion $\widetilde{j}$ is formalizable in the sense of Definition \ref{defn:formalizability}. 
\end{ex}

\section{Comparatively tractable rational models for polyhedral products}\label{sect:TractableModels}

The F\'elix--Tanr\'e rational model for a polyhedral product $(\underline{X}, \underline{A})^K$ depends on the choice of surjective models for the inclusions in the given tuple $(\underline{X}, \underline{A})$. While the model is complicated in general, the underlying algebra is adjustable in the sense of Theorem \ref{thm:tractableForm} below. In fact, we show that the underlying algebra of the model 
has a particular form which is regarded as a generalization of the rational model for a moment-angle complex; 
see Remark \ref{rem:RS}.

We recall the CDGA in (\ref{eq:MAC}). With this mind, we may call a CDGA $(A, d)$ a {\it Stanley--Reisner (SR) $K$-type}
if the underlying algebra $A$ is of the form 
\[
\bigotimes_{j=1}^m(\wedge V_j\otimes B_j))\big / \big.(b_{j_1}\cdots b_{j_s}\mid b_j \in B_j^+, \{j_1, ..., j_s\}\notin K), 
\]
where $B_j$ is a free commutative algebra. The term `K-' may be omitted if it is clear from the context. 

\begin{thm}\label{thm:tractableForm}
Each polyhedral product $(\underline{X}, \underline{A})^K$ has a SR type CDGA model; that is, there is a sequence of quasi-isomorphisms of CDGA's connecting $A_{PL}((\underline{X}, \underline{A})^K)$ and a Stanley-Reisner $K$-type CDGA. 
\end{thm}

\begin{proof}For each $j$, 
let $\varphi_j : \wedge W_j  \to \wedge V_j$ be a surjective model for the inclusion $\iota_j :  A_j \to X_j$. 
Since $\varphi_j$ is surjective, the vector space $W_j$ admits a decomposition 
$W_j\cong W_j'\oplus W_j''$ which satisfies the condition that $\varphi_j|_{W_j'} : W_j' \stackrel{\cong}{\to}  V_j$ is an isomorphism and $\varphi_j|_{W_j''} \equiv 0$. In fact, we choose  indecomposable elements $w_\lambda$ of $\wedge W_j$ so that $\varphi_j(w_\lambda)=v_\lambda$ for a basis 
$\{v_\lambda \}_{\lambda \in \Lambda}$ 
for $V_j$. Then, we have a decomposition 
\[
W_j\cong \Q\{ w_\lambda \mid \lambda \in \Lambda \}\oplus \Q\{w_\gamma'' \mid \gamma \in \Gamma\}
\]
with some index set $\Gamma$. 
Let $P(v_\lambda)$ be the polynomial on $v_\lambda$'s which represents $\varphi_j(w_\gamma'')$ in $\wedge V_j$. Putting 
$W_j'' := \Q\{ w_\gamma'' - P(w_\lambda) \}$, we have the decomposition required above.  
Theorem \ref{thm:F-T} yields the result. 
\end{proof}

We provide a more tractable SR type model for a polyhedral product of a fibre inclusion.  
For $1\leq j \leq m$, let $F_j \stackrel{\iota_j}{\longrightarrow} X_j \stackrel{p_j}{\longrightarrow} Y_j$ be a fibration with simply-connected base. Assume that $H^*(Y_j; \Q)$ is locally finite for each $j$. Then, a relative Sullivan model $\widetilde{p_j}$ for the map $p_j$ gives a commutative diagram of CDGA's
\[
\xymatrix@C20pt@R15pt{
\wedge W_j \ar[r]^-{\widetilde{p_j}} \ar[d]_-{\simeq}& ( \wedge V_j \otimes \wedge W_j, d_j) \ar[r]^-{\widetilde{\iota_j}} \ar[d]_-{\simeq}& (\wedge V_j, \overline{d_j}) \ar[d]_-{\simeq}\\
A_{PL}(Y_j) \ar[r]^-{p_j^*} &  A_{PL}(X_j) \ar[r]^-{\iota_j^*} & A_{PL}(F_j)  
}
\]
in which vertical maps are quasi-isomorphisms; see \cite[20.3 Theorem]{H}. 
The upper sequence is called a model for the fibration. 
It follows from the construction that $\widetilde{\iota_j}$ is a surjective model for 
$\iota_j$. If $\widetilde{p_j}$ is minimal, by definition, we see that
 $d(V_j) \subset (\wedge^{\geq 2}V_j)\otimes \wedge W_j  + \wedge V_j \otimes \wedge^+ W_j $  in the SR type CDGA.  
By virtue of Theorem \ref{thm:F-T}, we have 

\begin{prop}\label{prop:Inc} With the same notation as above, the polyhedral product $(\underline{X}, \underline{F})^K$ for the tuple of fibre inclusions 
$\iota_j$ has a SR type CDGA model of the form  
\[
\mathcal{M}((\underline{X}, \underline{F})^K):=\Big(\bigotimes_{j=1}^m(\wedge V_j\otimes \wedge W_j ))\big / \big.(b_{j_1}\cdots b_{j_s}\mid b_j \in W_j, \{j_1, ..., j_s\}\notin K), d\Big)
\]
for which $d(W_j) \subset \wedge W_j$ and $d(V_j) \subset (\wedge^{\geq 2}V_j)\otimes \wedge W_j  + \wedge V_j \otimes \wedge^+ W_j $. 
\end{prop}

\begin{rem}
The model in Proposition \ref{prop:Inc} is not a Sullivan model in general. However, 
if we construct a Sullivan model by using the model, then for example, we may obtain information on the rational homotopy group of 
$(\underline{X}, \underline{F})^K$; see Example \ref{ex:SU(n)} below. 
\end{rem}

\begin{ex}\label{ex:LX}
(i) Let $S^1 \to ES^1 \to BS^1$ be the universal $S^1$-bundle and $K$ be a simplicial complex with $m$ vertices. 
Then, we have a model for the bundle of the form $\wedge (dx) \to \wedge(dx)\otimes \wedge(x) \stackrel{\widetilde{\iota}}{\to} \wedge(x)$,  where 
$\widetilde{\iota}$  is the canonical projection and $\deg x = 1$. It follows from Proposition \ref{prop:Inc} that 
$
\mathcal{M}((ES^1, S^1)^K) \cong (\wedge (x_1, ..., x_m)\otimes SR(K), d)
$
where $d(x_j) = dx_j$; see Remark \ref{rem:RS}. Observe that $\mathcal{Z}_K(D^2, S^1)\simeq \mathcal{Z}_K(ES^1, S^1) \simeq_w (ES^1, S^1)^K$; see \cite[page 33]{D-S} for the first homotopy equivalence. 


(ii) 
Let $X$ be a simply-connected space and $LX$ the free loop space, namely the space of maps from $S^1$ to $X$ endowed with compact-open topology. The rotation action of $S^1$ on the domain of maps in $LX$ induces an $S^1$-action on the free loop space. 
Thus we have the Borel fibration
$LX \stackrel{i}{\to} ES^1 \times_{S^1}LX \stackrel{p}{\to} BS^1$.  We write $(LX)_h$ for the Borel construction $ ES^1 \times_{S^1}LX$. Let $(\wedge V, d)$ be the minimal model for $X$. 
Then, the result \cite[Theorem A]{V-B1} asserts that the sequence 
\[
\wedge (t) \stackrel{\widetilde{p}}{\to} (\wedge (t) \otimes \wedge (V\oplus \overline{V}), \delta) \stackrel{\widetilde{i}}{\to} (\wedge(V\oplus \overline{V}), \delta')
\]
is a model for the Borel fibration, where $\delta'(v) = d(v)$, $\delta'(\overline{v}) = -sd(v)$ and  $\delta u = \delta' (u) + ts(u)$ for $u \in V\oplus \overline{V}$. 
The map $\widetilde{i}$ is the projection and hence a surjective model for $i$. Thus Proposition \ref{prop:Inc} enables us to obtain a F\'elix--Tanr\'e model for the 
polyhedral product $((LX)_h, LX)^K$ of the form  
$
(\otimes_{i=1}^m(\wedge (V_i\oplus \overline{V}_i) \otimes SR(K), \otimes_i \delta_i).
$

(iii) We can apply Theorem \ref{thm:F-T} to an explicit surjective model for an inclusion. 
Let $X$  be a space as in (ii) and $(\wedge V, d)$ a minimal model for $X$. Then, the projection $(\wedge (V\oplus \overline{V}, \delta') \to 
(\wedge V, d)$ is a surjective model for 
the inclusion $X \to LX$ defined by assigning the constant loop at $x$ to a point $x$.  
In fact, the inclusion is a section of the evaluation map 
$ev_0 : LX \to X$ at zero. The inclusion $(\wedge V, d) \to (\wedge (V\oplus \overline{V}, \delta')$ gives rise to a model for $ev_0$. 
By considering the rational homotopy, we have the result.
Thus, Theorem \ref{thm:F-T} allows us to construct a model for the 
polyhedral product $(LX, X)^K$ 
of the form 
\[
(\bigotimes_{i=1}^m \wedge(V_i\oplus \overline{V_i}))\big/ \big. (\overline{v}_{i_1}\cdots \overline{v}_{i_s} \mid \overline{v}_j\in \overline{V}_j,  
\{i_1, ..., i_s\} \notin K)
\]
for which $d(V_i)\subset \wedge V_i$. 
\end{ex}

Proposition \ref{prop:Inc} enables us to deduce the following result. 

\begin{cor}\label{cor:SS} Let $K$ be a simplicial complex with $m$ vertices. For $1\leq j \leq m$, let 
$F_j \to X_j \to Y_j$ be a fibration with simply-connected base $Y_j$. 
Then there is a first quadrant spectral sequence converging to $H^*((\underline{X}, \underline{F})^K; \Q)$ as an algebra with
\[
E_2^{*, *} \cong \big(\bigotimes_{j=1}^mH^*(F_j; \Q)\big)\otimes H^*((\underline{Y}, \underline{*})^K; \Q)
\] 
as a bigraded algebra, where $E_2^{p, q} \cong (\big(\bigotimes_{j=1}^mH^*(F_j; \Q)\big)\otimes H^p((\underline{Y}, \underline{*})^K; \Q))^{p+q}$.
\end{cor}

\begin{proof}With the same notation as in Proposition \ref{prop:Inc}, we give the CDGA 
$\mathcal{M}((\underline{X}, \underline{F})^K)$ a filtration associated with the degrees of elements in $\otimes_j\wedge W_j$. 
The filtration gives rise to the spectral sequence; see \cite[18(b) Example 2]{F-H-T}. 
\end{proof}


Let $HH_*(A_{PL}(X))$ denote the Hochschild homology of $A_{PL}(X)$. There exists an isomorphism $HH_*(A_{PL}(X))\cong H^*(LX; \Q)$ of algebras; see \cite{V-B1} and \cite[15(c) Example 1]{F-H-T}. Therefore, 
Example \ref{ex:LX} (ii) allows us to obtain the following result. 

\begin{cor}\label{cor:LX} Let $X$ be a simply-connected space. Then, 
there exists a first quadrant spectral sequence converging to the cohomology $H^*(((LX)_h, LX)^K; \Q)$ as an algebra 
with
\[
E_2^{*,*} \cong HH_*(A_{PL}(X))^{\otimes m} \otimes SR(K)
\]
as a bigraded algebra, where $\text{\em bideg}\ x = (0,\deg x)$ for $x \in HH_*(A_{PL}(X))$ and $\text{\em bideg}\ t_i = (2, 0)$ for the generator $t_i \in SR(K)$. 
\end{cor}



\begin{rem} Let $F_j \stackrel{i_j}{\to} X_j \stackrel{p_j}{\to} B_j$ be a fibration for each $1 \leq j \leq m$. 
Suppose further that each $(X_j, F_j)$ is a pair of CW-complexes. Then, 
the F\'elix and Tanr\'e rational model for $(\underline{X}, \underline{F})^K $ in Proposition \ref{prop:Inc}
associated with the fibre inclusions is nothing but the relative Sullivan model for the pullback 
\[
\xymatrix@C25pt@R12pt{
(\underline{F}, \underline{F})^K \ar@{=}[r] \ar[d] & \Pi^m  \underline{F} \ar[d] \\
 (\underline{X}, \underline{F})^K \ar[r]  \ar[d]&  \Pi^m  \underline{X} \ar[d] \\
 (\underline{B}, \underline{*})^K \ar[r]  &  \Pi^m  \underline{B}
}
\]
which is introduced in \cite[Lemma 2.3.1]{D-S}. In fact, this follows from \cite[20.6]{H}.
\end{rem}

\section{A rational model for the polyhedral product of a pair of Lie groups}\label{sect:Liegroups}

In this section, we consider a more explicit model for the polyhedral product $(G, H)^K$ for a pair of a Lie group 
and a closed subgroup corresponding to arbitrary simplicial complex $K$. 
In particular, we have a manageable SR type model for $(G, H)^K$. Indeed, the rational model is determined by the image of the characteristic classes of 
$BG$ by the map $(Bi)^*: H^*(BG; \Q) \to  H^*(BH; \Q)$ for the inclusion $i : H \to G$; see Proposition \ref{prop:PolyProd-model} for more details of the model. 
By using the model, we prove Theorem \ref{thm:Rational_homotopy}. 

Let $G$ be a connected Lie group and $H$ a closed connected subgroup of $G$.  
Let $H\stackrel{i}{\to} G \stackrel{\pi}{\to} G/H$ be the principal $H$-bundle.  
In order to obtain a rational model for  $(G, H)^K$, we first construct an appropriate surjective model for the fibre inclusion $i$. 

Consider the fibration  $EH \to EH\times_H G \stackrel{q}{\to} G/H$ associated with the bundle $\pi$. Since $EH$ is contractible, it follows that the map $q$ is 
a weak homotopy equivalence. Moreover, we have a homotopy pullback diagram  
\begin{eqnarray}\label{eq:Apullback}
\xymatrix@C20pt@R12pt{
G \ar[d]_\iota \ar[r]^-{=}& G \ar[d]\\ 
EH\times_H G \ar[d] \ar[r]^-h & EG \ar[d]^{p_G} \\
BH \ar[r]_-{Bi} & BG,   
}
\end{eqnarray}
where vertical sequences are the associated fibration and the universal $G$-bundle, respectively, and $h$ is a map defined by $h([x, g])= Ei(x)g$. 
There exists a model for the universal bundle of the form 
\[
\xymatrix@C20pt@R18pt{
(\wedge V_{BG}, 0) \ar[r] \ar[d]_-{\simeq}^{m_{BG}}& ( \wedge V_{BG} \otimes \wedge P_G, d) \ar[r] \ar[d]_-{\simeq}^{m_{E_G}}& (\wedge P_G, 0) \ar[d]_-{\simeq}\\
A_{PL}(BG) \ar[r]^-{p_G^*} &  A_{PL}(E_G) \ar[r] & A_{PL}(G),  
}
\]
such that  $d(x_i) = y_i$ and $m_{E_G}(x_i) = \Psi_i$ for $x_i \in \wedge V_G$, where $\Psi_i \in A_{PL}(E_G)$ with $d\Psi_i = p_G^* m_{BG} (y_i)$. Then, 
by applying the pushout construction \cite[Proposition 15.8]{F-H-T} to the model of the bundle $p_G$, we have a model 
\[
\xymatrix@C20pt@R18pt{
(\wedge V_{BH}, 0) \ar[r] \ar[d]_-{\simeq}& ( \wedge V_{BH} \otimes \wedge P_G, d) \ar[r]^-{\overline{\iota}} \ar[d]_-{\simeq}^m& (\wedge P_G, 0) \ar[d]_-{\simeq}^-{m_G}\\
A_{PL}(BH) \ar[r] &  A_{PL}(EH\times_H G) \ar[r]^-{A_{PL}(\iota)} & A_{PL}(G)  
}
\]
of the fibration of the left hand side in the diagram (\ref{eq:Apullback})  in which $d(x_i) = (Bi)^*y_i$. Furthermore, the maps $\pi$, $\iota$ and $q$ mentioned above fit in the commutative diagram
\[
\xymatrix@C20pt@R18pt{
G \ar[rd]_-\iota \ar[r]^-\simeq & EH\times G \ar[d] \ar[r]^-\simeq & G \ar[d]^{\pi} \\
 & EH\times_H G \ar[r]_-\simeq^-q & G/H,    
}
\]
where horizontal arrows are (weak) homotopy equivalences. Thus, the Lifting lemma \cite[Proposition 14.6]{F-H-T} implies that a Sullivan model (\cite[15(a)]{F-H-T}) for $\iota$ is regarded as that for $\pi$. Consider a commutative diagram 
\[
\xymatrix@C20pt@R18pt{
(\wedge V_{BH}\otimes \wedge P_G, d) \ar[d]^-m_-\simeq \ar[r]^-j &(\wedge V_{BH}\otimes \wedge P_G \otimes \wedge P_H, \partial) \ar[r]^-\gamma_-{\simeq}& (\wedge P_G, 0) \ar[dl]^-{m_G}_-\simeq\\
A_{PL}(EH\times_H G) \ar[r]^-{A_{PL}(\iota)} &  A_{PL}(G)   
}
\]
of CDGA's in which $j$ is an extension, $\gamma$ is the projection and the differential $\partial$ is defined by $\partial(u_i)= t_i$ for $u_i \in P_H$, $t_i \in V_{BH}$ and 
$\partial(x_i) = (Bi)^*(y_i)$ for $x_i \in P_G$. Observe that $H^*(\wedge P_G, 0) \cong H^*(G ; \Q)$ and $\gamma$ is a quasi-isomorphism. Then, we see that the projection 
$\gamma : (\wedge V_{BH}\otimes \wedge P_G \otimes \wedge P_H, \partial)  \to (\wedge P_H, 0)$ 
is  a surjective model for  the inclusion $H \to G$. Thus, Proposition \ref{prop:Inc} yields the following result.

\begin{prop}\label{prop:PolyProd-model}
One has a rational model of the form 
\begin{eqnarray}\label{eq:PolyProd-model}
((\wedge P_H)^{\otimes m}\otimes \big((\wedge V_{BH}\otimes \wedge P_G)^{\otimes m}\big/ \big. I(K)\big), \partial)
\end{eqnarray}
for the polyhedral product $(G, H)^K$, where $I(K)$ denotes the Stanley--Reisner ideal generated by elements in 
$(\wedge V_{BK}\otimes \wedge P_G)^{\otimes m}$. 
\end{prop}

\begin{ex}\label{ex:maximal_rank_subgroups}
With the same notation as above, suppose further that $\text{rank} \ \!G = \text{rank} \ \!H = N$.  Then, the sequence $(Bi)^*(y_j)$ for $j = 1, ..., N$ is regular. This enables us to deduce that $G/H$ is formal. There exists a sequence of quasi-isomorphisms 
\[
\xymatrix@C15pt@R18pt{
A_{PL}(G/H) & \wedge V_{BH}\otimes \wedge P_G =:{\mathcal M} \ar[l]_-\simeq \ar[r]^-{\simeq}_-u & H^*(G/H) = 
(\wedge V_{BH}/((Bi)^*(y_i),  d=0). 
}
\]
The naturality (Proposition \ref{prop:naturality}) of the rational model in Theorem \ref{thm:F-T} gives rise to a commutative diagram
\[
\xymatrix@C12pt@R18pt{
\text{lim}_{\sigma \in K} A_{PL}((G/H, \ast)^\sigma) &
\text{lim}_{\sigma \in K}\widetilde{D}^\sigma \ar[l]_-{\Phi}^-{\simeq} \ar[d]_{\simeq}^{u_1}& (\bigotimes^m\mathcal{M})/ I(K)=:B_1 \ar[l]_-{\alpha}^-{\simeq}  \ar[d]^{u_2}\\
 A_{PL}((G/H, \ast)^K) \ar[u]^-{\eta}_-{\simeq} & \text{lim}_{\sigma \in K}H^*(G/H)^\sigma  & \big(\bigotimes^m H^*(G/H)\big)/ I(K)=:B_2, \ar[l]_-{\alpha}^-{\simeq}} \\
\]
where 
$H^*(G/H)^\sigma := \bigotimes_{i \in \sigma} H^*(G/H)^\sigma  \otimes \bigotimes_{i \notin \sigma} \Q$, 
$u_1$ and $u_2$  are maps of CDGA's induced by $u$; see the sequence (\ref{eq:The_sequence}).  By virtue of \cite[Proposition 4.8]{P-R}, we see that 
the map $u_1$ is a quasi-isomorphism. Thus, 
the commutativity implies that $u_2$ is also a quasi-isomorphism. 
Observe that $I(K) = J(K)$ in the case that we deal with; see Section \ref{sect:F-T}. 
We consider the pushout diagram of $\ell$ along $u_2$ 
\begin{eqnarray}\label{eq:u_ell}
\xymatrix@C15pt@R18pt{
B_1 \ar[d]^{\simeq}_{u_2} \ar[r]^-\ell & ((\wedge P_H)^{\otimes m}\otimes \big(\bigotimes^m\mathcal{M} \big)\big/ \big. I(K), \partial) 
\ar[d]^{\widetilde{u_2}} \\
B_2  \ar[r]^-{\widetilde{\ell}} & ((\wedge P_H)^{\otimes m}\otimes \big(\bigotimes^m H^*(G/H) \big)\big/ \big. I(K), \partial)=:C \\
} 
\end{eqnarray}
where $\ell$ is the KS-extension induced by the rational model for $(G, H)^K$ in (\ref{eq:PolyProd-model}); see \cite[Chapter 1]{H} for a KS-extension. 
 It follows from \cite[Lemma 14.2]{F-H-T} that $\widetilde{u_2}$ is a quasi-isomorphism and hence $C$ is also a rational model for $(G, H)^K$. 

In particular, for the unitary group $U(n)$, the maximal torus $T$ and every  simplicial complex $K$ with $m$ vertices, we have a rational model for 
$(U(n), T)^K$ of the form 
\[
((\wedge (x_i, ..., x_n))^{\otimes m}\otimes \big(\bigotimes_{i=1}^m \Q[t_1, ..., t_n]/(\sigma_1, ..., \sigma_n) \big)\big/ \big. I(K), \partial),
\]
where $\partial (x_i) = t_i$ and $\sigma_k$ denotes the $k$th elementary symmetric polynomial. 
\end{ex}

\begin{ex}\label{ex:SU(n)} Let $K$ be arbitrary simplicial complex with $m$ vertices. 
By virtue of Propositions \ref{prop:naturality} and \ref{prop:PolyProd-model}, 
we see that the projection $q : (SU(n), SU(k))^K \to  (SU(n)/SU(k), *)^K$ admits a model given by 
\[
 \widetilde{q} : \big(\wedge (x_{k+1}, .., x_n)^{\otimes m}/I(K)), 0 \big) \to 
 \big(\wedge (x_2, .., x_k)^{\otimes m}\otimes (\wedge (x_{k+1}, .., x_n)^{\otimes m}/I(K)), 0 \big), 
\]
where $\widetilde{q}(x_i) = x_i$ for $k+1 \leq i \leq n$,  and  $\deg x_i = 2i -1$. 
Since the domain of $\widetilde{q}$ admits a Sullivan algebra, we can construct a KS-extesion for $\widetilde{q}$. 
Then, it follows from Lemma \ref{lem:replacement} that the projection $q$ is formalizable in the sense of Definition \ref{defn:formalizability}.

Suppose further that the $1$-skeleton of $K$ does not coincide with that of $\Delta^m$. 
Then, the minimal model for $(SU(n)/SU(k), *)^K$ has nontrivial differential whose quadratic part is also nontrivial; see \cite[pages 144-145]{F-H-T} for a way to construct a minimal model for a CDGA. Therefore, the result 
\cite[Theorem 21.6]{F-H-T} yields that the rational homotopy groups $\pi_*((SU(n)/SU(k), *)^K)_\Q$ and 
$\pi_*((SU(n), SU(k))^K)_\Q$ have nontrivial Whitehead products. Observe that the Whitehead product on $\pi_*((SU(n)/SU(k))_\Q$ vanishes. 
\end{ex}

Let $X$ be a pointed space and $\pi^*(X):= H^*(Q(\wedge V), d_0)$ the homology of the vector space of indecomposable elements of a Sullivan model $(\wedge V, d)$ for $X$, where $Q(\wedge V)$ is the vector space of indecomposable elements and $d_0$ denotes the linear part of the differential $d$. 
There is a natural map $\nu_X$ from $\pi^*(\wedge V)$ to $\text{Hom}(\pi_*(X), \Q)$ provided $\pi_*(X)$ is abelian. 
Moreover, $\nu_X$ is an isomorphism if $X$ is a nilpotent space whose fundamental group is abelian; see \cite[11.3]{B-G}. 

It follows from the proof of \cite[Proposition 15.13]{F-H-T} that the natural map $\nu_{( \ )}$ is compatible with the connecting homomorphisms of the dual to the 
homotopy exact sequence for a fibration and the homology exact sequence for $\pi^*( \ )$ if fundamental groups of spaces of the fibration are abelian. Then, 
by considering the middle vertical fibration $\mathcal{F}$ in (\ref{eq:fibrations}), we have 

\begin{lem}\label{lem:nu} 
For each space $X$ in the fibration $\mathcal{F}$, 
the map $\nu_X$ is an isomorphism. 
\end{lem}

\begin{proof}[Proof of Theorem \ref{thm:Rational_homotopy}]
We first observe that $(G/H, \ast)^K$ is simply-connected. This follows from the Seifert--van Kampen theorem. Then, we see that $\pi_1((G, H)^K)$ is an abelian group. 

In what follows, we consider a Sullivan model for $(G, H)^K$ with the same notation as in Example \ref{ex:maximal_rank_subgroups}. 
Let $\beta_0 : \bigotimes^m\mathcal{M}
\to \big(\bigotimes^m H^*(G/H)\big)/ I(K)=:B_2$ be the composite of the quasi-isomorphism $u_2: B_1\to B_2$ mentioned above and 
the projection $\bigotimes^m\mathcal{M} \to (\bigotimes^m\mathcal{M})/ I(K)$. Observe that $\mathcal{M}= 
\wedge (V_{BH}\oplus P_G)$. 
Extending $\beta_0$, we define a quasi-isomorphism 
\[
\beta :   A_1:=(\bigotimes^m\mathcal{M}) \otimes \wedge V= \wedge ((\oplus^m (V_{BH}\oplus P_G)) \oplus V) \stackrel{\simeq}{\longrightarrow} B_2.
\]

Let $d_0$ denote the linear part of the differential of $A_1$. In order to construct a minimal Sullivan model for $A_1$, we 
apply the procedure of the proof of \cite[Theorem 14.9]{F-H-T}. As a consequence, there exists an isomorphism 
$(\wedge W, d') \otimes \wedge (U\oplus dU) \cong A_1$ for which $\wedge (U\oplus dU)$ is a contractible CDGA, $(\wedge W, d')$ is minimal, $(\oplus^m (V_{BH}\oplus P_G)) \oplus V = U \oplus \text{Ker} \ \! d_0 = 
U \oplus d_0U \oplus W$ and $d_0(W) =0$. By the construction, we may assume that $\oplus^m (V_{BH}\oplus P_G) \subset W$. 
Then, we have a quasi-isomorphism $\beta' : (\wedge W, d') \stackrel{\simeq}{\to} B_2$ and a pushout diagram 
\begin{eqnarray}\label{eq:I}
\xymatrix@C20pt@R20pt{
(\wedge W, d') \ar[r]^-I \ar[d]_{\beta'}^{\simeq} &  ((\wedge P_H)^{\otimes m} \otimes \wedge W, \partial) \ar[d]^{\beta''} \\
B_2  \ar[r]_-{\widetilde{\ell}} &  (\wedge P_H)^{\otimes m} \otimes  B_2, 
}
\end{eqnarray}
in which $\partial(x_k) = t_k \in V_{BH}$ for $x_k \in P_H$ and $I$ is the canonical inclusion. 
Therefore, it follows that the map $\beta''$ is a quasi-isomorphism.  Moreover, we see that the bottom right CDGA in the square above is nothing but the CDGA $C$ in Example \ref{ex:maximal_rank_subgroups}. 

By applying Lemma \ref{lem:replacement} repeatedly to the diagram obtained by combining the diagrams (\ref{eq:u_ell}) with (\ref{eq:I}) and to a commutative square 
 given by the naturality of maps in (\ref{eq:The_sequence}), we have a commutative diagram 
\begin{eqnarray}\label{eq:sequences}
\xymatrix@C25pt@R20pt{
(\wedge W, d') \ar[r]^-I \ar[d]_{\simeq} &  ((\wedge P_H)^{\otimes m} \otimes \wedge W, \partial) \ar[d]^{\simeq} \\
A_{PL}((G/H, \ast)^K)  \ar[r]_{A_{PL}(q)} &  A_{PL}((G, H)^K).
}
\end{eqnarray}

Recall that $\oplus^m(V_{BG}\oplus P_G)$ is a subspace of $W$. Then, 
we may write $(\wedge V_{BH}\otimes \wedge P_G)^{\otimes m}\otimes \wedge W'$ for $\wedge W$. Thus, the upper sequence in the diagram (\ref{eq:sequences}) gives rise to a short exact sequence of complexes 
\[
\xymatrix@C12pt@R15pt{
0 & (\oplus^m P_H, 0) \ar[l] & (
(\oplus^m (P_H\oplus V_{BH}\oplus P_G)) \oplus W' , \partial_0) \ar[l] &   
(W, 0), \ar[l] & 0 \ar[l]
}
\]
in which the linear part $\partial_0$ of $\partial$ satisfies the condition that $\partial_0 : \wedge P_H \to \wedge V_{BH}$, $\partial_0(t_k) = y_k$ and $\partial_0|_{P_G \oplus W'} = 0$. The last equality follows from the assumption that $(Bi)^*x_k$ is decomposable for each $k$. 
The homology long exact sequence is decomposed into a short exact sequence of the form 
\[
\xymatrix@C12pt@R15pt{
0 &H(\oplus^m (P_H\oplus V_{BH}\oplus P_G)) \oplus W' , \partial_0) \ar[l] &   
(W, 0) \ar[l] & (\oplus^m P_H, 0) \ar[l]_-{d_0'} & 0, \ar[l]
}
\]
where $d_0'$ denotes the connecting homomorphism. In fact, the linear map $d_0'$ coincides with the composite $\oplus^m P_H \stackrel{\partial_0}{\longrightarrow}  \oplus^m (P_H\oplus V_{BH}\oplus P_G)) \oplus W' \stackrel{pr}{\longrightarrow} W$, where $pr$ is the projection; see the proof of \cite[Proposition 15.13]{F-H-T}. Thus, Lemma \ref{lem:nu} yields the result. 
\end{proof}

\begin{rem}
Let $G$ be a compact Lie group and $H$ be a closed subgroup for which $G/H$ is simply connected and $(Bi)^*(x_k)$ is {\it indecomposable} in $H^*(BH; \Q)$ for some generator $x_k$ of $H^*(BG; \Q)$. Then, we see that the connecting homomorphism 
$$\partial_* : \pi_*((G/H, \ast)^K)_\Q \to \pi_{*-1}(\Pi^m H)_\Q$$ is not surjective. The connecting homomorphism is natural with respect to maps between spaces. Then, in order to prove the fact, it suffices to show that the connecting homomorphism $\partial^* : \pi^*(H) \to \pi^{*+1}(G/H)$ is not injective; see the diagram (\ref{eq:fibrations}). To this end, we show that the map $i^* : \pi^*(G) \to \pi^*(H)$ is non trivial. 

Recall the surjective model 
$\rho : (\wedge V_{BH}\otimes \wedge P_G \otimes \wedge P_H, \partial)  \to (\wedge P_H, 0)$ 
for the inclusion $H \to G$ used in the construction of the model (\ref{eq:PolyProd-model}). Suppose that $(Bi)^*(x_k) = \sum_i \lambda_i t_i+$(decomposable element) for some generator $x_k$ in $H^*(BG; \Q)$, where $\lambda_i \neq 0$ for some $i$.  Then, it follows that $x_i +\sum_i \lambda_i u_i$ is a cocycle in the cochain complex $(Q(\wedge V_{BH}\otimes \wedge P_G \otimes \wedge P_H), \partial_0)$ and 
\[
i^*(x_i +\sum_i \lambda_i u_i ) = \sum_i\lambda_i u_i  \neq 0
\]
for 
$i^* = H(Q(\rho)) : H(Q(\wedge V_{BH}\otimes \wedge P_G \otimes \wedge P_H), \partial_0) \to H(Q(\wedge P_H), 0) = P_H$. 

For example, we see that 
$\partial_* : \pi_*((U(n)/T, \ast)^K)_\Q \to \pi_{*-1}(\Pi^m T)_\Q$ is not surjective for a maximal torus $T$ of $U(n)$. 
\end{rem}

\section{The formality of a compact toric manifold}\label{sect:formality}

We prove the following result by using the commutative diagram (\ref{eq:pullbacks}). 

\begin{thm}\label{thm:formality}
Every compact toric manifold $X_\Sigma$ is formal.
\end{thm}

This result is proved in \cite{P-R, BMM}; see also \cite[Theorem 8.1.10]{B-P}. The proof of \cite[Proposition 3.1]{BMM} indeed uses the algebra structure of the cohomology of the toric manifold. We apply the F\'elix--Tanr\'e model for $DJ(K)$ in order to prove the fact. 

We also use a result due to Baum concerning a characterization of a regular sequence. 

\begin{prop}\label{prop:Baum}{\em (\cite[3.5 Proposition]{Baum})} Let $A$ be a connected commutative algebra and $a_1, ..., a_t$ elements of $A^{> 0}$. Set $\Lambda = \K[x_1, ..., x_t]$ with $\deg x_i = \deg a_i$ and consider $A$ to be a $\Lambda$-module by means of the map $f : \Lambda \to A$ defined by $f(x_i) = a_i$. Then the following are equivalent: 

\begin{itemize}
\item[\em (i)] $a_1, ..., a_t$ is a regular sequence. 
\item[\em (ii)] $\text{\em Tor}^{-1, *}_\Lambda(\K, A) = 0$. 
\item[\em (iii)] $\text{\em Tor}^{-j, *}_\Lambda(\K, A) = 0$ for all $j\geq 1$. 
\item[\em (iv)] $A$ is a projective $\Lambda$-module. 
\item[\em (v)] As a $\Lambda$-module $A$ is isomorphic to $\Lambda \otimes (A/(a_1, ..., a_t))$.
\end{itemize}

\end{prop}

The following result gives a rational model for the toric manifold $X_\Sigma$ in the proof of Theorem \ref{thm:formality}. 

\begin{lem} \label{lem:formalizability}
The map $(B\rho) \circ q : DJ(K) \to BL'$ in the diagram (\ref{eq:pullbacks}) is formalizable; see Definition \ref{defn:formalizability} and the paragraph after the diagram (\ref{eq:pullbacks}). 
\end{lem}

\begin{proof} 
The result follows from the same argument as in Example \ref{ex:DJ}. 
\end{proof}

\begin{proof}[Proof of Theorem \ref{thm:formality}] Let $\{ v_j\}_{j=1}^m$ be the set of $1$-dimensional cones of the fan $\Sigma$ of $n$ dimension. Each $v_i$ is in the lattice $N$ of $\R^n$ which defines the fan $\Sigma$. Then, it follows from the construction of the diagram (\ref{eq:pullbacks}) that $H^*(BL) \cong \Q[t_1', ..., t_n']$ as an algebra. Observe that $\dim L = \dim \Sigma =n$. Moreover, we see that for $i = 1, ..., n$, 
\[
(B\rho)^* (t_i') = \sum_{j=1}^m \langle m_i, v_j \rangle t_j, 
\]
where $t_j$ denotes the generator of $H^*(BG)\cong \Q[t_i, ..., t_m]$ and $m_i$ is the dual basis for $M:=\text{Hom}(N, {\mathbb Z})$. 
The F\'elix--Tanr\'e model for $DJ(K)$ is of the form $(SR(K)=\Q[t_i, ..., t_m]/I(K), 0)$ for which $q^*(t_j) = t_j$ for $j = 1, ..., m$. Consider the pushout construction of models (\cite{H, F-H-T}) for the pullback (\ref{eq:pullbacks}). 
Then, by Lemma \ref{lem:formalizability} and \cite[Proposition 2.3.4]{Vigue}, 
we have a rational model for $X_\Sigma$ of the form 
\[
 C:= (\wedge (x_1, ..., x_n) \otimes SR(K), d(x_i) = q^*(B\rho)^*(t_i')= \sum_{j=1}^m \langle m_i, v_j \rangle t_j), 
\]
where $\deg x_i = 1$. This also computes the torsion functor $\text{Tor}_{H^*(BL)}^{*,*}(H^*(DJ(K)), \Q)$ 
if we assign a bidegree $(-1, 2)$ to each $x_i$. 
The result \cite[Theorem 12.3.11]{CLS} asserts that $H^{\text{odd}}(X_\Sigma; \Q) =0$. This implies that 
$\text{Tor}_{H^*(BL)}^{-1,*}(H^*(DJ(K)), \Q)=0$. It follows from Proposition \ref{prop:Baum} that 
$q^*(B\rho)^*(t_1'), ..., q^*(B\rho)^*(t_n')$ is a regular sequence in $SR(K)$. Thus, we have a quasi-isomorphism 
$$f : C \to SR(K)\big/ (d(x_i) ; i = 1, ..., n )= H^*(X_\Sigma; \Q)$$ defined by $f(t_j) = t_j$ and $f(x_i) =0$. 
\end{proof}


\begin{rem}
We can also obtain the rational cohomology of the compact toric manifold by using the Eilenberg-Moore spectral sequence for the pullback 
(\ref{eq:pullbacks}). In fact, it follows from the computation of the spectral sequence that, as algebras, 
\[
H^*(X_\Sigma) \cong \text{Tor}_{H^*(BL)}(H^*(DJ(K)), \Q) \cong 
SR(K)\Big/ \big(\sum_{j=1}^m \langle m_i, v_j \rangle t_j \big).
\]
\end{rem}

One might be aware that the consideration above for the polyhedral product $({\mathbb C}, {\mathbb C}^*)^K$ is applicable to more general one, for example, 
$(EG, G)^K$ for a connected Lie group $G$. In fact, for a simplicial complex $K$ with $m$ vertices, we have (homotopy) pullback diagrams
\begin{eqnarray}\label{eq:pullbacks2}
\xymatrix@C20pt@R15pt{
X_{K, (G, H)} :=E(\Pi^m G) \times_H  (EG, G)^K \ar[r]^-p\ar[d]_{\pi}&  E(\Pi^m G)/H \ar[r]\ar[d]& EL\ar[d] \\
E(\Pi^m G) \times_{\Pi^m G} (EG, G)^K \ar[r]_-q&  B(\Pi^m G) \ar[r]_{B\rho}& BL, 
}
\end{eqnarray}
where $H$ is a normal (not necessarily connected) closed subgroup of $(\Pi^m G) $ and $L = (\Pi^m G) /H$. The result \cite[Lemma 2.3.2]{D-S} yields that the natural map $E(\Pi^m G) \times_{\Pi^m G} (EG, G)^K \stackrel{\simeq}{\to} (BG, \ast)^K$ is a homotopy equivalence. Then, we have 

\begin{thm}\label{thm:BorelCoh} Suppose that $H^{\text{\em odd}}(X_{K, (G, H)}; \Q)=0$. Then $X_{K, (G, H)}$ is formal. 
\end{thm}

Theorem \ref{thm:partial_quotients} asserts that the condition in Theorem \ref{thm:BorelCoh} is satisfied only for the toric manifold $M$ among partial quotients associated with $M$. 

\begin{cor}\label{cor:Rational_rigidity} With the same notation as above, suppose that $H^{\text{\em odd}}(X_{K, (G, H)}; \Q)=0$ and 
$H^*(X_{K, (G, H)}; \Q)\cong H^*(X_{K', (G', H')}; \Q)$. Then $X_{K, (G, H)}\simeq_\Q X_{K', (G', H')}$ if the spaces are nilpotent. 
\end{cor}

\begin{rem}\label{rem:Model_X}
Suppose that $H^*(BL) \cong \Q[t_1', ..., t_n']$. Under the same assumption as in Theorem \ref{thm:BorelCoh}, we see that $X_{K, (G, H)}$ admits a rational model of the form 
\[
(\wedge (x_1, ..., x_n) \otimes (\bigotimes^m H^*(G; \Q)\big/ I(K)), d(x_i) = q^*(B\rho)^*(t_i')) 
\]
in which $q^*(B\rho)^*(t_1'), ..., q^*(B\rho)^*(t_n')$ is a regular sequence. 
\end{rem}

We observe that, for a compact smooth toric manifold $X_\Sigma$, there is a homotopy equivalence $X_\Sigma \simeq X_{K, (({\mathbb C}^*)^m, H)}$ for which $K$ is a simplicial complex associated with the fan $\Sigma$ and 
$\mathcal{Z}_K({\mathbb C}, {\mathbb C}^*)/H$ is Cox's construction for $X_\Sigma$. Moreover, the toric manifold $X_\Sigma$ is simply connected and hence nilpotent. Thus, Corollary \ref{cor:Rational_rigidity} is regarded as an answer of  
the {\it rational cohomological rigidity problem} for toric manifolds

\section{The (non)formalizability of partial quotients}\label{sect:formalizability}
We begin by considering formalizability for toric manifolds.

\begin{prop}\label{prop:relatively_formalizable}
The map $\pi_H : X_\Sigma \to DJ(K)$ in (\ref{eq:pullbacks}) is formalizable. 
\end{prop}

\begin{proof}
By considering  the sequence (\ref{eq:The_sequence}) for $(X, A)= (BS^1, *)$, we have quasi-isomorphisms connecting 
$A_{PL}(DJ(K))$ with $SR(K)$ the Stanley--Reisner algebra. 
We can construct a minimal model $\phi : \wedge W \to SR(K)$ for $SR(K)$ 
so that $W= \Q\{t_1, ...,  t_m\}\oplus V$, where $t_1, ..., t_m$ give the generators of $SR(K)$, $\phi(V) = 0$ and $V= V^{\geq 2}$. The Lifting lemma (\cite[Proposition 12.9]{F-H-T}) yields a quasi-isomorphism $\phi' : \wedge W \to A_{PL}(DJ(K))$. 
Thus, in particular, the Davis--Januszkiewicz space $DJ(K)$ is formal. The pushout construction in the proof of Theorem \ref{thm:formality} gives rise to a commutative diagram
\[
\xymatrix@C20pt@R15pt{
A_{PL}((DJ(K)) \ar[r]^{A_{PL}(\pi_H)} & A_{PL}(X_\Sigma) \\
\wedge W \ar[r]^-i \ar[u]_{\simeq}^{\phi'} &  \wedge (x_1, ..., x_n)\otimes \wedge W,  \ar[u]^{\simeq}
}
\]
where $i$ is a KS-extension. 
Moreover, we have a commutative 
diagram of CDGA's
\[
\xymatrix@C20pt@R15pt{
\wedge W \ar[r]^-i \ar[d]^{\simeq}_\phi &  \wedge (x_1, ..., x_n)\otimes \wedge W \ar[d]_{\simeq}^f \\
H^*(DJ(K); \Q)  \ar[r]_{(\pi_H)^*} & H^*(X_\Sigma; \Q) 
}
\]
in which $\phi$ is a quasi-isomorphism defined by $\phi(t_j)=t_j$ and $\phi|_{V} \equiv 0$, the map $i$ is an extension and  $f$ is 
the quasi-isomorphism given in the end of the proof of Theorem \ref{thm:formality}.  
\end{proof}

\begin{proof}[Proof of Theorem \ref{thm:partial_quotients}] Let $\mathcal{Z}_K$ denote the space $\mathcal{Z}_K({\mathbb C}, {\mathbb C}^*)$. 
Suppose that $H = H'$. Then the partial quotient $\mathcal{Z}_K/H'$ is nothing but the toric manifold $X_\Sigma$. 
Then, the result \cite[Theorem 12.3.11]{CLS} implies the assertion (ii).  

We recall the proof of Theorem \ref{thm:formality}.  Then, we have a rational model 
\begin{eqnarray}\label{eq: C'}
\hspace{1cm}
 C':= (\wedge (x_1, ..., x_l) \otimes SR(K), d(x_i) = q^*(B\rho)^*(t_i') ) 
\end{eqnarray}
for $\mathcal{Z}_K/H'$. 
Under the assumption (ii), by Proposition \ref{prop:Baum}, we see that the sequence $d(x_1), ..., d(x_l)$ is regular. 
Thus the same argument as in the proof of Proposition \ref{prop:relatively_formalizable} yields (iii). 

Suppose that $H'$ is a connected proper subgroup of $H$. We show that $\pi_{H'}$ is not formalizable. 
We may replace three spaces $\mathcal{Z}_K$, the tori $H$ and $H'$ acting the moment-angle manifold with the polyhedral product 
$(D^2, S^1)^K$, a compact Lie group $T^k$ and its subtorus with an appropriate integer $k$, respectively. 
Assume that the fan $\Sigma$ has $m$ rays and hence $K$ is a simplicial complex with $m$ vertices. If the fan is of dimension $n$, then we have an exact sequence 
$1 \to T^k \to T^m \stackrel{\rho}{\to} T^n \to 1$ via Cox's construction of the toric manifold $X_\Sigma$. 
With the same notation as in the proof of Theorem \ref{thm:formality}, since $d(x_1), ..., d(x_n)$ is a regular sequence, it follows from Proposition \ref{prop:Baum} that 
$
H^*(DJ(K)) \cong \Q[t_1, ..., t_n]\otimes H^*(X_\Sigma)
$
as a $\Q[t_1, ..., t_n]$-module. 

For a proper subgroup $H'$ of $T^k$, the quotient $L'$ of the inclusion $H' \to T^m$ is the torus of dimension $l$ greater than $n$. 
Consider the rational model (\ref{eq: C'}) for $\mathcal{Z}_K/H'$. 
We assume that $d(x_1), ..., d(x_l)$ is a regular sequence. 
Then, by the same argument as in the proof of Theorem \ref{thm:formality} with the diagram (\ref{eq:pullbacks}), we have  
$H^*(DJ(K)) \cong \Q[t_1, ..., t_l]\otimes H^*(\mathcal{Z}_K/H')$ as 
a $\Q[t_1, ..., t_l]$-module. By considering the Poincar\'e series of $H^*(DJ(K))$ in two ways, we have an equality 
\[
\frac{1}{\Pi^{l-n}(1-t^2)} P_1(t) =  P_2(t),
\]
where $P_1(t)$ and  $P_2(t)$ are the Poincar\'e series of  $H^*(\mathcal{Z}_K/H')$ and $H^*(\mathcal{Z}_K/H)\cong H^*(X_\Sigma)$, respectively. Since the partial quotients are manifolds of finite dimensions, it follows that $P_1(t)$ and  $P_2(t)$ are polynomials. This contradicts the equality above and hence 
$d(x_1), ..., d(x_l)$ is not a regular sequence.  

Suppose that the map $\pi_{H'} : \mathcal{Z}_K/H' \to DJ(K)$ is formalizable. 
By virtue of  \cite[Proposition 2.3.4]{Vigue}, we have a commutative diagram 
\[
\xymatrix@C20pt@R12pt{
A_{PL}((DJ(K)) \ar[r]^{A_{PL}(\pi_{H'})} & A_{PL}(\mathcal{Z}_K/H') \\
\wedge W \ar[r]^-i \ar[d]^{\simeq}_\phi \ar[u]^{\phi'}_{\simeq}&  \wedge (x_1, ..., x_l)\otimes \wedge W \ar[d]^{\simeq}_\eta \ar[u]_{\simeq}\\
H^*(DJ(K); \Q)  \ar[r]_{(\pi_{H'})^*} & H^*(\mathcal{Z}_K/H'; \Q).
}
\]

Consider the fibration $\mathcal{Z}_K \to EG\times_{H'}\mathcal{Z}_K \stackrel{p}{\to} (EG)/H'$ which fits in the diagram (\ref{eq:pullbacks}). The argument in \cite[4.1]{D-S} enables us to conclude that $\mathcal{Z}_K$ is $2$-connected; see also \cite[Proposition 4.3.5]{B-P}. Thus, 
the homotopy exact sequence of the fibration above yields that 
$\mathcal{Z}_K/H' \simeq EG\times_{H'}\mathcal{Z}_K$ is simply connected and hence $H^1(\mathcal{Z}_K/H')=0$. 

By Lemma \ref{lem:formalizability} and \cite[Proposition 2.3.4]{Vigue}, 
we see that $\text{Tor}^{*}_{\Q[t_1, ..., t_l]}(\Q, SR(K))\cong \text{Tor}^{*}_{ \Q[t_1, ..., t_l]}(\Q, \wedge W)$. This implies that the spectral sequence converging to the torsion group $\text{Tor}^{*}_{ \Q[t_1, ..., t_l]}(\Q, \wedge W)$ with 
$E_2^{*,*}\cong \text{Tor}^{*,*}_{\Q[t_1, ..., t_l]}(\Q, SR(K))$ collapses at the $E_2$-term. Then, this fact allows us to obtain a sequence 
\[
\xymatrix@C20pt@R16pt{
\text{Tor}^{-1,*}_{P}(\Q, SR(K))\cong E_0^{-1, *}  & F^{-1}\text{Tor}^{*-1}_{P}(\Q, \wedge W) \ar[l]_-{p} \ar[r]^-{\iota}& 
\text{Tor}^{*-1}_{P}(\Q, \wedge W).
}
\]
Here $\{F^j\}$ denotes the filtration associated to the spectral sequence, $P$ is the polynomial algebra $ \Q[t_1, ..., t_l]$, $p$ and $\iota$ are the canonical projection and the inclusion, respectively. 
Since $d(x_1), ..., d(x_l)$ is not a regular sequence, it follows form Proposition \ref{prop:Baum} that there is a non-exact cocycle  
$$w = \sum_{j=1}^l u_j x_j - z$$ in 
$F^{-1}\text{Tor}^{*}_{ \Q[t_1, ..., t_l]}(\Q, \wedge W)$, where $u_i$ and $z$ are in $\wedge W$. 
Observe that the torsion algebra $\text{Tor}^{*}_{ \Q[t_1, ..., t_l]}(\Q, \wedge W)$ is isomorphic to the cohomology $H^{*}( \wedge (x_1, ..., x_l)\otimes \wedge W, d)$ as an algebra. 
We see that $\deg x_j = 1$ and then $\eta(x_j) = 0$ for each $j$. 
The element $z$ is of odd degree and in the image of the map $i$. Thus, since $H^{\text{odd}}(DJ(K); \Q)=0$, it follows that 
$$H^*(\eta)([w]) =[\eta(w)] = \eta(w)= \sum_{j=1}^l \eta(u_j) \eta(x_j) - (\pi')^*\phi(z) = 0,$$ which is a contradiction. 
\end{proof}

\medskip
\noindent
{\it Acknowledgements.}
The author thanks Grigory Solomadin and Shintaro Kuroki for many valuable discussions on toric manifolds and partial quotients without which 
he would not be able to attain Theorem \ref{thm:partial_quotients}. The author is also grateful to the referee for the careful reading of the paper and for valuable comments and suggestions. 

\appendix 
\section{A lifting lemma} 
In this section, we describe an algebraic result obtained by the Lifting lemmas \cite[Lemma 12.4 and Proposition 14.6]{F-H-T}.

\begin{lem}\label{lem:replacement}
For a commutative diagram (\ref{eq:U}) below with a Sullivan algebra $A_1$ and a KS-extension $I$, 
one has a commutative diagram (\ref{eq:U2}) in which $\widetilde{u_i}$ is a quasi-isomorphism 
if $u_i$ is for $i = 1$ or $2$.

\begin{minipage}{7cm}
\begin{eqnarray}\label{eq:U}
\xymatrix@C30pt@R15pt{
A_1 \ar[d]_{u_1} \ar[r]^-{I} & A_1\otimes \wedge W_1 \ar[d]^{u_2} \\
B_2  \ar[r]^-{\ell_2} & C_2 \\
B_1  \ar[u]_{\simeq}^{v} \ar[r]^-{\ell_1} & C_1 \ar[u]^\simeq _{v'} 
} 
\end{eqnarray}
\end{minipage}
\hspace{-1cm}
\begin{minipage}{7cm}
\begin{eqnarray}\label{eq:U2}
\xymatrix@C30pt@R15pt{
A_1 \ar[d]_{\widetilde{u_1}} \ar[r]^-{I} & A_1\otimes \wedge W_1 \ar[d]^{\widetilde{u_2}} \\
B_1   \ar[r]^-{\ell_1} & C_1 
} 
\end{eqnarray}
\end{minipage}
\end{lem}

\begin{proof} By applying the surjective trick (\cite[page 148]{F-H-T}) to $v$, we have a diagram 
\begin{eqnarray*}
\xymatrix@C25pt@R15pt{
& B_1 \ar[d]^{\simeq}_{v} \ar[r]^-{\ell_1} & C_1
\ar[d]^{v'}_{\simeq} \\
& B_2  \ar[r]^-{\ell_2} & C_2 \\
A_1 \ar@{.>}[r]_-\xi \ar[ru]|(.5)\hole^(.4){u_1} & B_1 \otimes\wedge S \ar@/^3.0pc/[uu]^\lambda_{\simeq} \ar@{->>}[u]_{\simeq}^{\widetilde{v}} 
\ar[r]^-{\ell_1\otimes 1} & C_1 \otimes\wedge S
\ar[u]_{\widetilde{v'}}^{\simeq}  \ar@/_3.0pc/[uu]_{\lambda'}^{\simeq}
} 
\end{eqnarray*}
of solid arrows in which three squares are commutative. We observe that $S \cong B_2\oplus dB_2$ and that $\widetilde{v'}$ is defined by 
$\widetilde{v'}(c_1) = v'(c_1)$ for $c_1 \in C_1$ and $\widetilde{v'}(s) = \ell_2(s)$ for $s \in S$. Since $A_1$ is a Sullivan algebra, the Lifting lemma \cite[Lemma 12.4]{F-H-T} enables us to obtain the map $\xi$ which fits in the commutative triangle.  Thus we have a commutative diagram of solid arrows
\[
\xymatrix@C25pt@R20pt{
A_1 \ar[d]_I  \ar[r]^-\xi &  B_1 \otimes\wedge S \ar[r]^{\ell_1\otimes 1}& C_1 \otimes\wedge S \ar[d]^{\widetilde{v'}}_\simeq\\
A_1\otimes \wedge W_1 \ar[rr]_{u_2} \ar@{.>}[rru]_{\overline{u_2}}& & C_2. 
}
\]
Since the map $I$ is a KS-extension, by using \cite[Proposition 14.6]{F-H-T}, we have a dotted arrow  
$\overline{u_2}$ which makes the upper triangle commutative and the lower triangle commutative up to homotopy relative to $A_1$. Define 
$\widetilde{u_1} := \lambda \circ \xi$ and $\widetilde{u_2} := \lambda' \circ \overline{u_2}$. Then, we have the commutative diagram (\ref{eq:U2}). 
By the construction of the map $\widetilde{u_i}$, we see that $\widetilde{u_i}$ is a quasi-isomorphism 
if $u_i$ is  for $i = 1$ or $2$. 
\end{proof}

\end{document}